\documentclass[twoside]{amsart}%
\usepackage{amssymb}
\usepackage{amsmath,amsthm}
\usepackage{amsfonts}
\usepackage{graphicx}
\usepackage{amsmath}
\usepackage[all]{xy}
\usepackage{array}
\usepackage{xcolor}
\usepackage{multicol}
\usepackage{hyperref}

\providecommand{\U}[1]{\protect\rule{.1in}{.1in}}
\newtheorem{theorem}{Theorem}

\newtheorem{definition}[theorem]{Definition}

\newtheorem{lemma}[theorem]{Lemma}

\newtheorem{remark}[theorem]{Remark}

\date{}

\begin{document}


\title{On the Absoluteness of $\aleph_1$-Freeness}

\author{Daniel Herden, Alexandra V. Pasi}
\address{
Department of Mathematics,
Baylor University,
Sid Richardson Building,
1410 S.4th Street,
Waco, TX 76706, USA}
\email{daniel\_herden@baylor.edu, alexandra\_pasi@baylor.edu}


\maketitle

\begin{abstract}
$\aleph_1$-free groups, abelian groups for which every countable subgroup is free, exhibit a number of interesting algebraic and set-theoretic properties. In this paper, we give a complete proof that the property of being $\aleph_1$-free is absolute; that is, if an abelian group $G$ is $\aleph_1$-free in some transitive model $\textbf{M}$ of ZFC, then it is $\aleph_1$-free in any transitive model of ZFC containing $G$. 
The absoluteness of $\aleph_1$-freeness has the following remarkable consequence: an abelian group $G$ is $\aleph_1$-free in some transitive model of ZFC if and only if it is (countable and) free in some model extension.
This set-theoretic characterization will be the starting point for further exploring the relationship between the set-theoretic and algebraic properties of $\aleph_1$-free groups. In particular, this paper will demonstrate how proofs may be dramatically simplified using model extensions for $\aleph_1$-free groups. 
\end{abstract}

\section{Introduction}

$\aleph_1$-free groups, abelian groups whose countable subgroups are free, are the subject of significant study in abelian group theory. The Baer-Specker group, the direct product of countably many copies of $\mathbb{Z}$, provides a simple example of an $\aleph_1$-free group (which is not free), \cite{Baer, Specker}. The class of $\aleph_1$-free groups exhibits a high degree of algebraic complexity, as evidenced by the test of ring-realization: any ring with free additive structure can be realized as the endomorphism ring of some $\aleph_1$-free group \cite{CG, DG}.

$\aleph_1$-free groups also play an important role in the Whitehead Problem, as every Whitehead group can be shown to be $\aleph_1$-free \cite{Stein}. The Whitehead problem poses the question of whether there exist any non-trivial abelian groups $A$ with $\operatorname{Ext}(A,\mathbb{Z})=0$. Such groups are known as Whitehead groups. Clearly any free abelian group is a Whitehead group, and the Whitehead problem asks, in other words, whether every Whitehead group is free. Shelah proved that the Whitehead problem is undecidable from ZFC \cite{Eklof, Shelah}, by showing that every Whitehead group is free in $\mathbf{L}$ and that there exist non-free Whitehead groups assuming $\mathrm{MA+\lnot CH}$. This constituted the first time that a seemingly purely algebraic problem was demonstrated to be independent from ZFC. $\aleph_1$-free groups have a number of interesting set-theoretic properties which contribute to their unique positioning in such undecidability results \cite{EM,GT}.

We intend to investigate and utilize these set-theoretic properties of $\aleph_1$-free groups as part of a larger project exploring a family of forcings involving $\aleph_1$-free groups. One such property of particular interest is that $\aleph_1$-freeness is absolute: that is, an abelian group $G$ is $\aleph_1$-free in a transitive model $\mathbf{M}$ of ZFC if and only if it is $\aleph_1$-free (in $\mathbf{V}$). The absoluteness of $\aleph_1$-freeness has been informally acknowledged within the community as a ``folk theorem,'' but for our purposes in establishing and exploring various forcings, we require a complete formal proof of the absoluteness of $\aleph_1$-freeness, which we will establish in the first part of this paper. It has also conventionally been understood that this absoluteness result implies that any $\aleph_1$-free group can be made free by forcing the cardinality of the group to collapse to countable. In this sense, we can think of $\aleph_1$-freeness as being equivalent to ``potential freeness" in some appropriate model extension. 

However, in making these ``folkloric'' results rigorous, we can see that there are some formalistic issues, both mathematical and meta-mathematical, that arise which need to be addressed. In particular, we note that while we use, for the sake of convenience, the forcing convention of starting from a countable transitive ground model, this interpretation appears to significantly limit the characterization of $\aleph_1$-freeness as ``potential freeness'' to those groups which are countable in $\mathbf{V}$. There are, however, alternative formalizations of forcing such as the Boolean-valued model approach or the Naturalist account of forcing which do not require the ground model to be countable and thus do not impose the restriction that the $\aleph_1$-free groups in question be countable in $\mathbf{V}$. It should be remarked that while these Boolean-valued models are not transitive, Mostowski collapse can be performed in order to generate transitive models which are isomorphic to these models. For an in-depth discussion of some of the mathematical and philosophical subtleties surrounding these different approaches to forcing, we refer the interested reader to Hamkins' paper on the set-theoretic multiverse \cite{Hamkins}. As yet another alternative philosophical resolution to the restriction to countable models (and thus countable groups), one might even accept the epistemic possibility espoused by Pruss in \cite{Pruss} that all infinite sets are countable.

Regardless of the philosophical approach taken to forcing, we can use the ``potential freeness'' characterization to establish a novel proof technique for theorems concerning $\aleph_1$-free groups. Notably, this technique can even be carried out for $\aleph_1$-free groups of arbitrary cardinality using the countable-transitive-model formalization of forcing. The formalistic concerns regarding forcing with a countable transitive model are addressed in a way analogous to that with which independence results are generally justified, by taking countable fragments of the appropriate parts of ZFC. The proof technique which we present here significantly shortens classical algebraic arguments by moving to a model in which the $\aleph_1$-free groups in question become free. We give two examples of such proofs at the end of this paper. A similar technique using forcing extensions has been utilized by Baumgartner and Hajnal in \cite{BH}, and by Todor\u{c}evi\'{c} and Farah in \cite{TF} to construct sets with certain combinatorial properties. However, the use of generic extensions to prove general results about algebraic objects as demonstrated in this paper is unique in the literature.

The results in this paper lay the groundwork for a larger project involving forcing with $\aleph_1$-free groups, and the novel proof technique provides a key tool for working with such forcings. In particular, the absoluteness of $\aleph_1$-free groups allows us to give a complete characterization of when an $\aleph_1$-free group of cardinality $\aleph_1$ can be forced to be free with cardinal preservation \cite{BHP}. 

\section{Background} \label{sec2}

We begin in this section by introducing and collecting some general results on absoluteness. We refer the reader to \cite[Chapter IV.2--5]{Kunen} for further proofs and details.

We first define relativization, which allows us to explore the notion of truth in a given model $\textbf{M}$.

\begin{definition}
Let $\emph{\textbf{M}}$ be any class. Then for any formula $\phi$, we define $\phi^\emph{\textbf{M}}$, the \emph{relativization of $\phi$ to $\emph{\textbf{M}}$}, inductively as follows:
\begin{enumerate}
\item $(x=y)^\emph{\textbf{M}}$ is $x=y$.
\item $(x \in y)^\emph{\textbf{M}}$ is $x \in y$.
\item $(\phi \land \psi)^\emph{\textbf{M}}$ is $\phi^\emph{\textbf{M}} \land \psi^\emph{\textbf{M}}$.
\item $(\neg \phi)^\emph{\textbf{M}}$ is $\neg(\phi^\emph{\textbf{M}})$.
\item $(\exists x \ \phi)^\emph{\textbf{M}}$ is $\exists x\ (x \in \emph{\textbf{M}} \land \phi^\emph{\textbf{M}})$.
\end{enumerate}
\end{definition}

\begin{definition}
Let $\emph{\textbf{M}}$ be any class. For a sentence $\phi$, ``\,\emph{$\phi$ is true in $\emph{\textbf{M}}$}'' means that $\phi^\emph{\textbf{M}}$ is true. For a set of sentences $S$, ``\emph{\,$S$ is true in $\emph{\textbf{M}}$}'' means that each sentence in $S$ is true in $\emph{\textbf{M}}$.
\end{definition}

We are now ready to give a definition for absoluteness.

\begin{definition}
Let $\phi$ be a formula with free variables $x_1,\ldots,x_n$. If $\emph{\textbf{M}} \subseteq \emph{\textbf{N}}$, $\phi$ is \emph{absolute for $\emph{\textbf{M}}, \emph{\textbf{N}}$} if and only if
\[
\forall x_1,\ldots,x_n \in \emph{\textbf{M}}\, (\phi^\emph{\textbf{M}}(x_1,\ldots,x_n) \longleftrightarrow \phi^\emph{\textbf{N}}(x_1,\ldots,x_n)).
\]
We say that $\phi$ is \emph{absolute for $\emph{\textbf{M}}$} if and only if $\phi$ is absolute for $\emph{\textbf{M}}, \emph{\textbf{V}}$. That is,
\[
\forall x_1,\ldots,x_n \in \emph{\textbf{M}}\, (\phi^\emph{\textbf{M}}(x_1,\ldots,x_n) \longleftrightarrow \phi(x_1,\ldots,x_n)).
\]
\end{definition}

Intuitively, the absoluteness of $\phi$ for $\textbf{M}, \textbf{N}$ means that $\phi(x_1,\ldots,x_n)$ is true in $\textbf{M}$ if and only if it is true in $\textbf{N}$. Note that if $\phi$ is absolute for both $\textbf{M}$ and $\textbf{N}$, and $\textbf{M} \subseteq \textbf{N}$, then $\phi$ is absolute for $\textbf{M}, \textbf{N}$.
The following lemma states that absoluteness is preserved under logical equivalence.

\begin{lemma} \label{logical equivalence}
Suppose $\emph{\textbf{M}} \subseteq \emph{\textbf{N}}$, and both $\emph{\textbf{M}}$ and $\emph{\textbf{N}}$ are models for a set of sentences $S$ such that
\[
S \vdash \forall x_1,\ldots, x_n\, (\phi(x_1,\ldots,x_n) \longleftrightarrow \psi(x_1,\ldots,x_n)).
\]
Then $\phi$ is absolute for $\emph{\textbf{M}}, \emph{\textbf{N}}$ if and only if $\psi$ is absolute for $\emph{\textbf{M}}, \emph{\textbf{N}}$.
\end{lemma}

The following definition introduces a family of formulas, the $\Delta_0$ formulas, which is foundational to our results on absoluteness.

\begin{definition} \label{bound}
A formula is $\Delta_0$ if it is built inductively according to the following rules:
\begin{enumerate}
\item $x \in y$ and $x=y$ are $\Delta_0$.
\item If $\phi, \psi$ are $\Delta_0$, then $\neg \phi$, $\phi \land \psi$, $\phi \lor \psi$, $\phi \to \psi$ and $\phi \leftrightarrow \psi$ are $\Delta_0$.
\item If $\phi$ is $\Delta_0$, then $\exists x\, (x \in y \land \phi)$ and $\forall x\, (\neg(x \in y) \lor \phi)$ are $\Delta_0$.
\end{enumerate}
\end{definition}


We use $\exists x \in y \,\, \phi$ as abbreviation for $\exists x\, (x \in y \land \phi)$ and $\forall x \in y \,\, \phi$ as abbreviation for $\forall x\, (\neg(x \in y) \lor \phi)$, and
we call $\exists x \in y$ and $\forall x \in y$ \emph{bounded quantifiers}. According to Definition \ref{bound}, a formula in which all quantifiers are bounded
is automatically $\Delta_0$. The next lemma connects $\Delta_0$ formulas to absoluteness.

\begin{lemma}
If $\emph{\textbf{M}}$ is transitive and $\phi$ is $\Delta_0$, then $\phi$ is absolute for $\emph{\textbf{M}}$.
\end{lemma}


In addition to the previous exposition, we need also to account for the absoluteness of defined notions which take the form of functions. This gives rise to the following definition.

\begin{definition} If $\emph{\textbf{M}} \subseteq \emph{\textbf{N}}$, and $F(x_1,\ldots,x_n)$ is a well-defined function both for $\emph{\textbf{M}}$ and $\emph{\textbf{N}}$, we say $F$ is absolute for $\emph{\textbf{M}}, \emph{\textbf{N}}$ if the formula $F(x_1,\ldots,x_n)=y$ is absolute for $\emph{\textbf{M}}, \emph{\textbf{N}}$. \\
More formally, suppose that $F(x_1,\ldots,x_n)$ was defined as the unique $y$ such that $\phi(x_1,\ldots,x_n,y)$. Then $F(x_1,\ldots,x_n)$ is a well-defined function for $\emph{\textbf{M}}, \emph{\textbf{N}}$ only if
\[\forall x_1,\ldots,x_n\, \exists!y\,\, \phi(x_1,\ldots,x_n,y)\]
is true in both $\emph{\textbf{M}}$ and $\emph{\textbf{N}}$. Assuming this, $F$ is absolute for $\emph{\textbf{M}}, \emph{\textbf{N}}$ if and only if $\phi$ is absolute.
\end{definition}

This definition allows us to make full sense of the following lemma, which states that absolute notions are closed under composition.

\begin{lemma}
Let $\emph{\textbf{M}} \subseteq \emph{\textbf{N}}$, and suppose that formula $\phi(x_1,\ldots,x_n)$ and functions $F(x_1,\ldots,x_n)$, $G_i(y_1,\ldots,y_m)$ $(i=1,\ldots,n)$ are absolute for $\emph{\textbf{M}}, \emph{\textbf{N}}$. Then so are the formula
\[
\phi(G_1(y_1,\ldots,y_m),\ldots,G_n(y_1,\ldots,y_m))
\]
and the function
\[
F(G_1(y_1,\ldots,y_m),\ldots,G_n(y_1,\ldots,y_m)).
\]
\end{lemma}

Using the definitions and results above, we can establish the absoluteness of a number of defined notions and formulas from set theory.
\begin{lemma}
The following are absolute for any transitive model $\emph{\textbf{M}}$ of $\rm{ZFC}$:
\begin{multicols}{2}
\begin{enumerate}
\item ordered pairs $(x,y)$,
\item set union $\bigcup x$,
\item set inclusion $x \subseteq y$,
\item $x$ is an ordinal,
\item $\alpha+\beta$, $\alpha \cdot \beta$ for ordinals $\alpha, \beta$,
\item $\omega$ and $(\mathbb{Z},+,\cdot)$.
\end{enumerate}
\end{multicols}
\end{lemma}

\begin{proof}
Note that $\mathbb{Z}$ is formally defined as a set of ordered pairs from $\omega \times \omega$.
The operations $+$ and $\cdot$ on $\mathbb{Z}$
are defined appropriately and are primarily determined by the ordinal arithmetic of $\omega$.
\end{proof}

The absoluteness of finite sets will be of particular importance.

\begin{lemma}
If $\emph{\textbf{M}}$ is a transitive model of $\rm{ZFC}$, then every finite subset of $\emph{\textbf{M}}$ is in $\emph{\textbf{M}}$, and
``\,$x$ is finite'' is absolute for $\emph{\textbf{M}}$.
\end{lemma}

Our last lemma in this section addresses the absoluteness of formulas which are built recursively over $\omega$ from absolute formulas.

\begin{lemma}
Suppose $F: \emph{\textbf{V}} \to \emph{\textbf{V}}$, and let $G: \omega \to \emph{\textbf{V}}$ be defined so that
\[
\forall n \in \omega\, [G(n)=F(G\upharpoonright n)],
\]
where $G\upharpoonright n$ denotes the restriction of $G$ to the domain $n=\{0,1,\ldots,n-1\}$.

Let $\emph{\textbf{M}}$ be a transitive model of $\rm{ZFC}$ and assume that $F$ is absolute for $\emph{\textbf{M}}$. Then also $G$ is absolute for $\emph{\textbf{M}}$.
\end{lemma}


\section{Results} \label{sec3}

We will now apply the absoluteness results in Section \ref{sec2} to some algebraic notions. After relating in Theorem \ref{aleph1-freeness} the $\aleph_1$-freeness of a group to the freeness of the pure subgroups generated by its finite subsets, we will be ready to establish the main result of this section, Theorem \ref{aleph1-free absolute}, namely the absoluteness of $\aleph_1$-freeness.

\subsection{Basic Absoluteness Results for Abelian Groups}

In this section, we collect some first basic absoluteness results on abelian groups. We will follow the algebraic convention of using $G$ as a shorthand to denote the abelian group $( G,+ )$.

\begin{lemma} Suppose $\emph{\textbf{M}}$ is a transitive model of $\operatorname{ZFC}$. Then ``\,$G$ is an abelian group'' is absolute for $\emph{\textbf{M}}$.
\end{lemma}

\begin{proof}
Suppose $( G,+ ) \in \textbf{M}$. Note the logical equivalence
\[
\text{``$G$ is an abelian group''} \iff \phi_1 \land \phi_2 \land \phi_3,
\]
where $\phi_1, \phi_2$ and $\phi_3$ denote the $\Delta_0$ sentences
\begin{enumerate}
\item $\forall x \in G\ \forall y \in G\ \forall z \in G\ (x+(y+z)=(x+y)+z)$,
\item $\exists u \in G\ ((\forall x \in G\ (x+u=x)) \land (\forall x \in G\ \exists y \in G\ (x+y=u)))$,
\item $\forall x \in G\ \forall y \in G\ (x+y=y+x)$. \qedhere
\end{enumerate}
\end{proof}

\begin{lemma} Suppose $\emph{\textbf{M}}$ is a transitive model of $\operatorname{ZFC}$ and $( G,+ ) \in \mathbf{M}$ is abelian. The defined notions ``\,$0_G$'' and ``\,$nx$'' $(n\in \mathbb{Z}, x\in G)$ are absolute for $\emph{\textbf{M}}$.

\begin{proof}
To see that ``\,$0_G$'' is an absolute defined notion, note that $0_G$ is uniquely defined by
\[z=0_G \iff (z\in G \land \forall x \in G\ (x+z=x)),
\]
where the right-hand side of the above equivalence is a $\Delta_0$ formula.

For the absoluteness of ``$nx$'', note that $nx$ for $n\ge 0$ is formally defined recursively on $n\in \omega$ by $0x=0_G$ and $nx=(n-1)x+x$ for all $n>0$. Absoluteness easily extends to $n<0$ as $nx$ is the additive inverse of $(-n)x$.
\end{proof}
\end{lemma}

\begin{lemma} Suppose $\emph{\textbf{M}}$ is a transitive model of $\operatorname{ZFC}$ and $( G,+ ) \in \mathbf{M}$ is abelian. Then ``\,$G$ is torsion-free'' is absolute for $\emph{\textbf{M}}$.
\end{lemma}

\begin{proof}
Note that
\[
\textnormal{``$G$ is torsion-free''} \iff \forall x \in G\ \forall n \in \omega\ ( nx=0_G \rightarrow (x=0_G \lor n=0)),
\]
where the right-hand side above is obtained by substituting absolute notions ``\,$0_G$'' and ``$nx$'' into a $\Delta_0$ sentence.
\end{proof}

\subsection{Establishing the Absoluteness of \texorpdfstring{$\aleph_1$-Freeness}{aleph1-Freeness}}

We are nearly ready to establish our main result concerning the absoluteness of $\aleph_1$-freeness. We will review finite rank pure subgroups and Pontryagin's Criterion, the main ingredients of the proof of Theorem \ref{aleph1-freeness}. The interested reader may refer to \cite{Fuchs} for more detail.

\begin{definition}
A subgroup $H$ of an abelian group $G$ is said to be a \emph{pure} subgroup if for any $x \in H$, $0\ne n \in \mathbb{Z}$, $n \mid x$ in $G$ implies $n \mid x$ in $H$. In particular, a subgroup $H$ of a torsion-free group $G$ is pure if and only if $x=ny$ implies $y \in H$ for all $x \in H, y\in G$ and $0\ne n \in \mathbb{Z}$.

The intersection of pure subgroups of a torsion-free group is again pure. Therefore if $G$ is a torsion-free abelian group, and $S$ is a subset of $G$, the intersection of all pure subgroups containing $S$ is the \emph{minimal pure subgroup containing $S$}, which we denote by $\langle S \rangle_*$.

Explicitly, we may write
\[
\langle S \rangle_*\! =\! \{y \in G\! \mid\! \exists n,n_1,\ldots ,n_m\! \in\! \mathbb{Z}, n\neq 0\, \exists s_1,\ldots ,s_m\!\! \in\! S\!:\! ny=n_1s_1+\ldots +n_ms_m \}.
\]
\end{definition}

\begin{lemma} \label{<x>_* absolute}
Suppose $\emph{\textbf{M}}$ is a transitive model of $\operatorname{ZFC}$. Suppose  $( G,+ ) \in \mathbf{M}$ is abelian and $S$ is a finite subset of $G$. Then ``\,$\langle S\rangle _*$'' is an absolute notion for $\emph{\textbf{M}}$.

\begin{proof}
Let $S=\{s_1,\ldots, s_m\}$.
Then we have the logical equivalence
\begin{flalign*}
& z = \langle S \rangle _* \iff \\
& \quad\bigg[ \! \Big[ \forall y \in z\ \big[y \in G
\land \exists n,n_1,\ldots ,n_m \in \mathbb{Z}\ (\neg (n=0) \land ny=n_1s_1+\ldots +n_ms_m)\big]\! \Big] & \\
& \land\! \Big[\forall y \in G\,
\big[(\exists n,n_1,\ldots ,n_m\! \in\! \mathbb{Z}\, (\neg (n=0)
\land ny=n_1s_1\!+\! \ldots\! +\! n_ms_m))\! \rightarrow y \in z\big]\! \Big]\! \bigg],
\end{flalign*}
where the statement on the right-hand side above involves only bounded quantifiers, logical symbols, ``$\mathbb{Z}$'', and various multiples of elements. 
\end{proof}
\end{lemma}

\begin{remark}
A more careful argument shows that ``\,$\langle S\rangle _*$'' is an absolute notion for any subset $S$ of $G$ in $\emph{\textbf{M}}$. Note also that for any finite set $S\in \emph{\textbf{V}}$ with $S\subseteq G$, we have $\langle S\rangle _*\in \emph{\textbf{M}}$.
\end{remark}

In order to establish our result on the absoluteness of $\aleph_1$-freeness,
we need a simple estimate for the torsion-free rank of $\langle S \rangle_*$. Recall that
the \emph{torsion-free rank $\operatorname{rk}_0(G)$} of a torsion-free abelian group $G$ is defined
as the size of a maximal linearly independent subset $S\subseteq G$.


\begin{lemma}[ZF] \label{lem:S and S'}
If S is a finite subset of a torsion-free abelian group, then
\[
\operatorname{rk}_0(\langle S \rangle_*) \leq |S|.
\]
\end{lemma}

\begin{proof}
Let $S=\{s_1,\ldots, s_m\}$ and choose a maximal linearly independent subset $S'$ of the finite set $S$. Thus $S'\cup \{s_i\}$ is linearly dependent for all $s_i \in S\setminus S'$, and $S \subseteq \langle S' \rangle_*$ follows. In particular, as $S$ is finite, there exists some $0\ne N \in \mathbb{Z}$ with $Ns_i \in \langle S' \rangle$ for all $s_i\in S$. 

Clearly, $\langle S' \rangle_* \subseteq \langle S \rangle_*$. Let $t \in \langle S \rangle_*$. Then there exist $n, n_1,\ldots ,n_m \in \mathbb{Z}$, $n\not=0$ such that $nt=\sum_{i=1}^{m} n_i s_i \in \langle S \rangle$. Hence $Nnt=\sum_{i=1}^{m} n_i (Ns_i) \in \langle S' \rangle$, which implies $\langle S \rangle_* \subseteq \langle S' \rangle_*$. Thus $\langle S \rangle_* = \langle S' \rangle_*$.

We wish to show $\operatorname{rk}_0(\langle S \rangle_*) = |S'|$. By way of contradiction, suppose there exists some $t \in \langle S \rangle_*$ such that $S' \cup \{t\}$ is linearly independent.
Then $t \in \langle S \rangle_* = \langle S' \rangle_*$. So there exists some $0\ne n \in \mathbb{Z}$ such that $nt \in \langle S' \rangle$, and $t$ is linearly dependent on $S'$, contradicting our assumption.
Thus, $S'$ is a maximal linearly independent subset of $\langle S \rangle_*$, and so $\operatorname{rk}_0(\langle S \rangle_*) = |S'| \leq |S|$.
\end{proof}

\begin{remark}
An alternative proof of Lemma \ref{lem:S and S'} uses the divisible hull \mbox{$\mathbb{Q}\otimes \langle S \rangle$} and $\operatorname{rk}_0(\langle S \rangle_*) = \dim_\mathbb{Q} (\mathbb{Q}\otimes \langle S \rangle) = \operatorname{rk}_0(\langle S \rangle)$. We have chosen a more elementary approach to better focus on the underlying aspects of set theory.

More generally, using the axiom of choice {\rm (AC)}, $\operatorname{rk}_0(\langle S \rangle_*) \leq |S|$ holds for any subset of a torsion-free group.
\end{remark}

We recall Pontryagin's Criterion next. For a proof, see \cite[Theorem 7.1]{Fuchs}.

\begin{theorem}[Pontryagin's Criterion]
A countable torsion-free abelian group is free if and only if each of its finite rank subgroups is free.
\end{theorem}

The theorem below gives a number of equivalent characterizations of $\aleph_1$-freeness. We will use the last of these alternative characterizations to prove the absoluteness of $\aleph_1$-freeness. Note that $\operatorname{rk}(G)$ denotes the rank of $G$, and $\operatorname{rk_0}(G)$ denotes its torsion-free rank. 

\begin{theorem} \label{aleph1-freeness}
Let $\emph{\textbf{M}}$ be a transitive model of $\operatorname{ZFC}$ and $G$ be an abelian group in $\emph{\textbf{M}}$. The following statements are equivalent:
\begin{itemize}
\item[$(i)$] $G$ is $\aleph_1$-free, that is, for all subgroups $H$ of $G$, if $|H| \leq \aleph_0$, then $H$ is free.
\item[$(ii)$] For all subgroups $H$ of $G$, if $\operatorname{rk}(H)$ is finite, then $H$ is free.
\item[$(iii)$] $G$ is torsion-free and for all pure subgroups $H$ of G, if $\operatorname{rk}(H)$ is finite, then $H$ is free.
\item[$(iv)$] $G$ is torsion-free and for all finite subsets $S$ of $G$, $\langle S \rangle_*$ is free.
\end{itemize}
\begin{proof}
For $(i) \rightarrow (ii)$, let $H$ be a subgroup of $G$ with $\operatorname{rk}(H)$ finite. By $(i)$, all cyclic subgroups of $G$ are free. Hence $G$ is torsion-free, and $H\subseteq G$ is torsion-free, too. This implies $\operatorname{rk}(H)=\operatorname{rk}_0(H)$ and $|H| \le \aleph_0 \cdot \operatorname{rk}(H)$. Hence $|H| \leq \aleph_0$, and $H$ is free.

The direction $(ii) \rightarrow (iii) \rightarrow (iv)$ is easy. For $(ii) \rightarrow (iii)$, note that with $(ii)$ every cyclic subgroup of $G$ is free, hence $G$ is torsion-free. For $(iii) \rightarrow (iv)$, note that if $S$ is finite, then $\langle S \rangle_*$ is of finite rank by Lemma \ref{lem:S and S'}.

To see that $(iv) \rightarrow (i)$, let $H$ be a subgroup of $G$ with $|H| \leq \aleph_0$. If $H=0$, $H$ is free, so suppose that $H$ is non-trivial. Then as $G$ is torsion-free, $|H|=\aleph_0$.
We wish to use Pontryagin's Criterion to show that $H$ is free, so let $K$ be a finite rank subgroup of $H$. Choose $S$ to be a maximal linearly independent subset of $K$. Then as $S$ is finite, $\langle S \rangle_*$ is free, and $K \subseteq \langle S \rangle_*$ is free, too. So by Pontryagin's Criterion, $H$ is free.
\end{proof}
\end{theorem}

Finally, we are ready to establish the absoluteness of $\aleph_1$-freeness below.

\begin{theorem} \label{aleph1-free absolute}
Suppose $\emph{\textbf{M}}$ is a transitive model of $\operatorname{ZFC}$, and $G$ is an abelian group in $\emph{\textbf{M}}$. Then ``\,$G$ is $\aleph_1$-free'' is absolute.
\begin{proof}
Let $\phi$ denote the statement
\[
G \textnormal{ is torsion-free } \land\ \forall S(( S \subseteq G \land S \textnormal{ is finite}) \rightarrow \langle S \rangle_* \textnormal{ is free}).
\]
By Theorem \ref{aleph1-freeness}, ``$G$ is $\aleph_1$-free'' is equivalent to $\phi$. To establish the absoluteness of $\aleph_1$-freeness, we will show that $\forall G \in \mathbf{M}\ (\phi ^ \textbf{M} \longleftrightarrow \phi)$.

We need to determine $\phi ^ \textbf{M}$ first. Recall that torsion-freeness, set inclusion, finiteness and ``\,$\langle S \rangle_*$'' (for finite subsets $S$ of a torsion-free abelian group) are absolute. Unfortunately, however, freeness is not absolute. Thus $\phi ^ \textbf{M}$ is the statement
\[
G \textnormal{ is torsion-free } \land\ \forall S\in \textbf{M}\ (( S \subseteq G \land S \textnormal{ is finite}) \rightarrow \langle S \rangle_* \textnormal{ is free$^\textbf{M}$}).
\]

For $\phi ^ \textbf{M} \rightarrow \phi$, let $S\in \textbf{V}$ such that $S \subseteq G$ and $S$ is finite, and assume $\phi ^ \textbf{M}$. Note that by our previous result on the absoluteness of finite sets, as $G \in \textbf{M}$, if $S$ is a finite subset of $G$ then $S \in \textbf{M}$. Thus, $\langle S \rangle_*$ has a basis in $\textbf{M}$, which is automatically a basis of $\langle S \rangle_*$ in $\textbf{V}$.\par
For $\phi \rightarrow \phi ^ \textbf{M}$, let $S\in \textbf{M}$ such that $S \subseteq G$ and $S$ is finite, and assume~$\phi$. Then  $\langle S \rangle_*$ is free in $\textbf{V}$. Choose a basis $B\in \textbf{V}$ of $\langle S \rangle_*$. Then, still in $\textbf{V}$, $\operatorname{rk}_0(\langle S \rangle_*)=|B|$ and $\operatorname{rk}_0(\langle S \rangle_*) \leq |S|$. So $|B| \leq |S|$, and thus $B\subseteq \langle S \rangle_* \subseteq G$ is a finite set. So by the absoluteness of finite sets, $B \in \textbf{M}$, and $B$ will witness that $\langle S \rangle_*$ is free in $\textbf{M}$. \par
Thus, $\aleph_1$-freeness is absolute.
\end{proof}
\end{theorem}


\section{Proofs with Model Extensions}

In this section, we discuss some major applications and consequences of the absoluteness of $\aleph_1$-freeness. We start with a general observation concerning the relationship between $\aleph_1$-freeness and freeness in different models of set theory. We then demonstrate how this observation can be turned into a quick and elegant routine for generating and simplifying proofs concerning $\aleph_1$-free groups.

We will repeatedly reference the use of forcing to collapse the cardinality of a given $\aleph_1$-free group $G$ to countable. Such a forcing may be defined by the partial order consisting of all injective functions from finite subsets of $\omega$ into $G$ which results in a model extension in which $G$ is countable. Thus, by the absoluteness of $\aleph_1$-freeness, $G$ is free in this model extension, as it is a countable subgroup of itself. For more detail on this type of forcing, we refer to Kunen~[3]. However, the technical details of forcing are not necessary in order to understand the applications below. Rather, we simply reference forcing as a method for producing a model extension with some required properties, namely with the property that $G$ is free in this model extension.\smallskip

The following result is another take on Theorem \ref{aleph1-free absolute} and highlights how $\aleph_1$-freeness as an initially algebraic property can be interpret and understood in the context of model extensions in set theory.

\begin{theorem} \label{absproof}
Let $\emph{\textbf{M}}$ be a transitive model of {\rm ZFC}, and $G$ an abelian group in~$\emph{\textbf{M}}$. Then the following are equivalent:
\begin{itemize}
    \item[$(i)$] $G$ is $\aleph_1$-free in $\emph{\textbf{M}}$.
    \item[$(ii)$] $G$ is $\aleph_1$-free in $\emph{\textbf{V}}$.
    \item[$(iii)$] $G$ is $\aleph_1$-free in any transitive model $\emph{\textbf{N}}$ with $G \in \emph{\textbf{N}}$.
\end{itemize}
In addition, if $\emph{\textbf{M}}$ is a countable transitive model of {\rm ZFC},
we may add to the above list of equivalent statements:
\begin{itemize}
\item[$(iv)$] $G$ is free in some generic extension $\emph{\textbf{N}}$ of $\emph{\textbf{M}}$.
\end{itemize}
\begin{proof}
The equivalence of $(i)$, $(ii)$ and $(iii)$ is an immediate consequence of the absoluteness of $\aleph_1$-freeness, Theorem \ref{aleph1-free absolute}. 

We have $(iv) \rightarrow (i)$ as $G$ is free in $\textbf{N}$ implies $G$ is $\aleph_1$-free in $\textbf{N}$, and thus by the absoluteness of $\aleph_1$-freeness, $G$ is $\aleph_1$-free in $\textbf{M}$. 
Finally, $(i) \rightarrow (iv)$ can be seen by letting $\textbf{N}$ be the generic extension obtained by collapsing the cardinality of $G$ to countable. 
\end{proof}
\end{theorem}

Theorem \ref{absproof} provides a new approach to proving statements about $\aleph_1$-free groups. To illustrate the utility of such an approach, we give a remarkably simple proof of the well-known transitivity of $\aleph_1$-free groups. We will be using countable transitive models of ZFC as is common convention for forcing arguments. It should be understood that countable transitive models only exist for finite lists of axioms and we will provide in Remark \ref{metath} an explanation of how our proof with countable transitive models translates into a formal proof within the metatheory. 

\begin{theorem} \label{ex1}
If $H$ and $G/H$ are $\aleph_1$-free for abelian groups $H\subseteq G$, then $G$ is $\aleph_1$-free.
\begin{proof} We will assume a countable transitive model $\textbf{M}$ with $H,G \in \textbf{M}$.
Let $\textbf{N}$ be the generic extension of $\textbf{M}$ produced by collapsing the cardinality of $G$ to countable. Then $G/H$ is countable and $\aleph_1$-free in $\textbf{N}$, thus it is free in $\textbf{N}$. In particular, we have $G^\textbf{N} \cong H \oplus G/H$, and as $H$ is also countable and free in $\textbf{N}$, $G$ is free in $\textbf{N}$. 
With Theorem \ref{absproof}, $(iv) \rightarrow (i)$, $G$ is $\aleph_1$-free in~$\textbf{M}$.
\end{proof}
\end{theorem}

In a nutshell, the absoluteness of $\aleph_1$-freeness lets us prove statements about $\aleph_1$-free groups by focusing entirely on the special case
of countable free groups. While we do provide a formal justification of why we can restrict our attention to the special case of countable groups, philosophically this is perhaps unsurprising from the perspective of all infinite sets being potentially countable \cite{Pruss}. We provide one more demonstration of this novel
proof technique.

\begin{theorem} \label{ex2}
Let $G$ be $\aleph_1$-free and let $H \subseteq G$ be a finite rank pure subgroup of $G$. Then $G/H$ is $\aleph_1$-free.
\begin{proof} 
Without loss of generality, assume $G \in \textbf{M}$ for some countable transitive model $\textbf{M}$ of ZFC. Suppose that $\textbf{N}$ is some generic model extension of $\textbf{M}$ in which $G$ is countable. Then $G$ is free in $\textbf{N}$. So in $\textbf{N}$, as $G$ is free, we can choose a basis $B$ of $G$.

Let $H$ be a finite rank pure subgroup of $G$ (recall that this is an absolute property). Then $H$ is free in $\textbf{N}$ of finite rank, so we can choose $B' \subseteq B$ finite with $H \subseteq \langle B' \rangle $.
Now $H$ is a pure subgroup of $G$, and thus, $H$ is pure in $\langle B' \rangle $. So $\langle B' \rangle / H$ is torsion-free, and as it is also finitely generated, $\langle B' \rangle / H$ is free by the fundamental theorem of abelian groups.
So $H$ is a direct summand of $\langle B' \rangle$, which is a direct summand of $\langle B \rangle = G$. Thus $G/H$ is free in $\textbf{N}$, and by the absoluteness of $\aleph_1$-freeness, $G/H$ is $\aleph_1$-free in $\textbf{M}$.
\end{proof}
\end{theorem} 

\begin{remark} \label{metath}
We will discuss a more formal argument for Theorem \ref{ex1} and Theorem \ref{ex2} to explain why we can restrict ourselves to countable transitive models $\emph{\textbf{M}}$ with $H,G \in \emph{\textbf{M}}$.

Let $\psi$ denote the first-order logical sentence which expresses the statement of Theorem~\ref{ex1} (or Theorem \ref{ex2}) and note that the proofs of Theorem~\ref{ex1} and Theorem~\ref{ex2} can be formalized using a finite list of axioms $\varphi_1,\ldots,\varphi_n$ of {\rm ZFC}. If $\psi$ were not provable on the basis of {\rm ZFC}, then G\"{o}del's Completeness Theorem implies the existence of a model for {\rm ZFC + $\neg \psi$}. In particular,
the finite list of axioms $\varphi_1,\ldots,\varphi_n, \neg \psi$ is consistent and  a standard procedure using the
Reflection Theorem, L\"{o}wenheim-Skolem Theorem, and Mostowski Collapse Lemma produces a countable transitive model $\emph{\textbf{M}}$ for $\varphi_1,\ldots,\varphi_n,$ $\neg \psi$.
In particular, in $\emph{\textbf{M}}$ we can find abelian groups
$G,H$ for which $\psi$ fails and going from $\emph{\textbf{M}}$ to
a generic model extension $\emph{\textbf{N}}$ where $G^\emph{\textbf{N}}$ is countable we can reproduce the proofs of Theorem~\ref{ex1} and Theorem~\ref{ex2} for a contradiction.
\end{remark} 


\section{Concluding Remarks}

Strategic use of the absoluteness of $\aleph_1$-free groups provides a new powerful tool for proving results concerning $\aleph_1$-free groups. In addition, this absoluteness implies that a group $G$ is $\aleph_1$-free in a countable transitive model \textbf{M} of {\rm ZFC} if and only if it is free in some generic model extension \textbf{N} of \textbf{M}. It is this observation which constitutes the starting point of our forthcoming series of papers constructing and exploring models of {\rm ZFC} related to $\aleph_1$-free groups. Further research in this area will illuminate the relationship between the algebraic and set-theoretic properties of $\aleph_1$-free groups, contributing to a deeper understanding of these groups and expanding the potential for their application.

\end{document}